\documentclass[11pt,a4paper,reqno]{amsart}
\usepackage{amsfonts}
\usepackage{amssymb}
\usepackage{hyperref}

\numberwithin{equation}{section}
\newtheorem{theorem}{Theorem}[section]
\newtheorem{lemma}[theorem]{Lemma}

\newtheorem{corollary}[theorem]{Corollary}
\newtheorem{proposition}[theorem]{Proposition}

\newtheorem{example}[theorem]{Example}
\newenvironment{conjecture}[1][Conjecture.\ ]{\textbf{#1}\it}{\rm}\def\eqref#1{(\ref{#1})}
\allowdisplaybreaks

\begin{document}
\author{Alexander Grigor'yan}
\address{School of Mathematical Sciences and LPMC, Nankai University, 300071
Tianjin, P. R. China and Department of Mathematics, University of Bielefeld,
33501 Bielefeld, Germany}
\email{grigor@math.uni-bielefeld.de}
\author{Yuhua Sun}
\address{School of Mathematical Sciences and LPMC, Nankai University, 300071
Tianjin, P. R. China}
\email{sunyuhua@nankai.edu.cn}
\author{Igor Verbitsky}
\address{Department of Mathematics, University of Missouri, Columbia,
Missouri 65211, USA}
\email{verbitskyi@missouri.edu}
\title{Superlinear elliptic inequalities on manifolds}
\thanks{\noindent Grigor'yan was supported by SFB1283 of the German Research
Council. Sun was supported by the National Natural Science Foundation of
China (No.11501303, No. 11871296, No. 11761131002), and also by the Fundamental
Research Funds for the Central Universities.}

\begin{abstract}
Let $M$ be a complete non-compact Riemannian manifold and let $\sigma $ be a
Radon measure on $M$. We study the problem of existence or non-existence of
positive solutions to a semilinear elliptic inequaliy
\begin{equation*}
-\Delta u\geq \sigma u^{q}\quad \text{in}\,\,M,
\end{equation*}%
where $q>1$. We obtain necessary and sufficent criteria for existence of
positive solutions in terms of Green function of $\Delta $. In particular,
explicit necessary and sufficient conditions are given when $M$ has
nonnegative Ricci curvature everywhere in $M$, or more generally when
Green's function satisfies the 3G-inequality.
\end{abstract}

\maketitle
\tableofcontents

\section{Introduction}

Let $M$ be a connected complete non-compact Riemannian manifold. Denote by $%
\mathcal{M}^{+}\left( M\right) $ the class of nonnegative Radon measures on $%
M$. In this paper we are concerned with the following problem: characterize $%
q>1$ and $\sigma \in \mathcal{M}^{+}(M)$ for which there exists a positive
solution $u\in C^{2}\left( M\right) $ to the following superlinear elliptic
inequality:
\begin{equation}
\Delta u+\sigma u^{q}\leq 0\quad \text{in}\,\,M,  \label{ineq}
\end{equation}%
where $\Delta $ is the Laplace-Beltrami operator on $M$.

If such a solution $u$ exists, then $u$ is a non-constant positive
superharmonic function on $M$, so that $M$ is non-parabolic. Hence, we can
assume without loss of generality that $M$ is non-parabolic. In particular,
the operator $\Delta $ on $M$ has a positive finite Green function (see \cite%
{G0}). Denote by $G\left( x,y\right) $ the minimal Green function.

Clearly, any $C^{2}$ non-negative solution $u$ of (\ref{ineq}) satisfies the
following integral inequality:
\begin{equation}
u(x)\geq \int_{M}G(x,y)\,[u(y)]^{q}d\sigma (y).  \label{int-ineq}
\end{equation}%
We also consider the integral inequality (\ref{int-ineq}) independently of (%
\ref{ineq}). By a solution of (\ref{int-ineq}) we mean any non-negative
l.s.c. function that satisfies (\ref{int-ineq}) for all $x\in M$.


In this paper, we give necessary and sufficient conditions for the existence
of positive solutions to (\ref{int-ineq}) and (\ref{ineq}) in terms of
certain properties of the Green function. Of course, any necessary condition
for (\ref{int-ineq}) will also be necessary for (\ref{ineq}). On the other
hand, if $\sigma $ has a smooth positive density with respect to $\mu $ then
the existence of a positive solution for (\ref{int-ineq}) implies that for (\ref%
{ineq}) (see Lemma \ref{Lemsmooth} below). Hence, in the rest of the paper
we concentrate on the integral inequality (\ref{int-ineq}) unless otherwise
specified.

Denote also by $\mu $ the Riemannian measure on $M$ and by $d$ the geodesic
distance. The geodesic balls on $M$ will be denoted by $B(x,r)=\{y\in
M:\,d(x,y)<r\}$, where $x\in M$ and $r>0$. In what follows, we assume
without loss of generality that
\begin{equation}
\int_{r_{0}}^{+\infty }\frac{tdt}{\mu (B(o,t))}<\infty  \label{cond-0}
\end{equation}%
for some/all $o\in M$ and $r_{0}>0$, since it is known that condition (\ref%
{cond-0}) is necessary for the non-parabolicity of $M$ (see \cite{G85}, \cite{Var}).

Our first result uses the following hypothesis:
\begin{equation}
G(x,y)\approx \int_{d(x,y)}^{+\infty }\frac{tdt}{\mu (B(x,t))},\quad x,y\in
M,  \tag{GLY}  \label{ly}
\end{equation}%
where the sign $\approx $ means that the ratio of the left- and right-hand
sides is bounded from above and below by two positive constants. For
example, estimate (\ref{ly}) holds if the Ricci curvature of $M$ is
non-negative, which follows from the heat kernel estimate of Li and Yau \cite%
{LY}.

More generally, (\ref{ly}) holds whenever the following two conditions are
satisfied:

\begin{enumerate}
\item the volume doubling condition: for all $x\in M$ and $r>0$
\begin{equation}
\mu (B(x,2r))\leq C\,\mu (B(x,r));  \tag{VD}  \label{D}
\end{equation}

\item the Poincar\'{e} inequality: for any ball $B=B(x,r)\subset M$ and any $%
f\in C^{2}\left( B\right) $,
\begin{equation}
\int_{B}|f-f_{B}|^{2}d\mu \leq C\,r^{2}\,\int_{B}|\nabla f|^{2}d\mu ,
\tag{PI}  \label{pi}
\end{equation}%
where $f_{B}$ stands for the mean value of $f$ on $B$ and $C$ is some
constant;
\end{enumerate}

\noindent(see \cite{G}, \cite{GSC}, \cite{LT}, \cite{LY}, \cite{SC}).

\begin{theorem}
\label{main} Assume that conditions {(\ref{D}) and} (\ref{ly}) are
satisfied. Then (\ref{int-ineq}) has a positive solution if and only if
there exist $o\in M$, $r_{0}>0$ and $C>0$ such that the following two
conditions hold:
\begin{equation}
\int_{r_{0}}^{+\infty }\left[ \int_{r}^{+\infty }\frac{tdt}{\mu (B(o,t))}%
\right] ^{q-1}\frac{\sigma (B(o,r))}{\mu (B(o,r))}rdr<\infty ,
\label{cond-int1}
\end{equation}%
and
\begin{equation}
\sup_{x\in B(o,r)}\left[ \int_{0}^{r}\frac{\sigma (B(x,s))}{\mu (B(x,s))}%
\,sds\right] \,\left[ \int_{r}^{+\infty }\frac{tdt}{\mu (B(o,t))}\right]
^{q-1}\leq C,  \label{cond-int2}
\end{equation}%
for all $r>r_{0}$.

Moreover, if (\ref{int-ineq}) has a positive solution then both (\ref%
{cond-int1}) and (\ref{cond-int2}) hold for all $o\in M$ and $r_{0}>0$ with $%
C=C\left( o,r_{0}\right) $.
\end{theorem}

\medskip
In particular, Theorem \ref{main} gives necessary and sufficient conditions
for the existence of a positive solution to (\ref{int-ineq})
on manifolds $M$ with nonnegative Ricci curvature.
\medskip

Consider now a special case $\sigma =\mu $, that is, the inequality
\begin{equation}
\Delta u+u^{q}\leq 0\quad \text{in}\,\,M.  \label{dx}
\end{equation}%
In this case (\ref{cond-int1}) clearly implies  (\ref{cond-int2}). Thus, a
necessary and sufficient condition for the existence of a positive solution
to (\ref{dx}) becomes
\begin{equation}
\int_{r_{0}}^{+\infty }\left[ \int_{r}^{+\infty }\frac{tdt}{\mu (B(o,t))}%
\right] ^{q-1}rdr<\infty ,  \label{cond-int1a}
\end{equation}%
for some $r_{0}>0$. Furthermore, (\ref{cond-int1a}) can be simplified as
follows.

\begin{corollary}
\label{cor-dx} Under the assumptions of Theorem \ref{main}, inequality (\ref%
{dx}) has a $C^{2}$ positive solution if and only if
\begin{equation}
\int_{r_{0}}^{+\infty }\frac{r^{2q-1}dr}{[\mu (B(o,r))]^{q-1}}<\infty ,
\label{cond-int1b}
\end{equation}%
for some $o\in M$ and $r_{0}>0$.
\end{corollary}

We remark that, for general coefficients $\sigma $, the local uniform bound (%
\ref{cond-int2}) provides an additional restriction in comparison to (\ref%
{cond-int1}), in particular in the special case $M={\mathbb{R}}^{n}$ where $%
\mu $ is Lebesgue measure (see \cite{KV}).

Let (\ref{D}) and (\ref{ly}) be satisfied on $M$. Assume in addition that,
for some $o\in M$ and large enough $r$,
\begin{equation}
\mu (B(o,r))\leq Cr^{\alpha }(\ln r)^{\frac{\alpha -2}{2}}(\ln \ln r)^{\frac{%
\alpha -2}{2}}(\ln \ln \ln r)^{\frac{\alpha -2}{2}}\cdot \cdot \cdot (%
\underbrace{\ln \cdot \cdot \cdot \ln }_{k}r)^{\frac{\alpha -2}{2}},
\label{mun}
\end{equation}%
where $\alpha >2$ and $k$ is a positive integer. It is clear that integral
in (\ref{cond-int1b}) diverges if and only if $q\leq \frac{\alpha }{\alpha -2%
}$. Hence, by Corollary \ref{cor-dx}, in the case $q\leq \frac{\alpha }{%
\alpha -2}$ there is no positive solution to (\ref{dx}).

{Condition (\ref{mun}) with }$k=1$ was considered previously in {\cite%
{GS1}. More precisely, \cite[Theorem 1.1]{GS1} says the following: if }$M$
is any connected complete manifold such that for some $o\in M$ and large
enough $r${\ } {%
\begin{equation}
\mu (B(o,r))\leq c\,r^{\alpha }(\ln r)^{\frac{\alpha -2}{2}},  \label{log}
\end{equation}%
then (\ref{dx}) has no positive solutions for any }$q\leq \frac{\alpha }{%
\alpha -2}${. Let us emphasize that the result of \cite[Theorem 1.1]{GS1}
does not require preconditions }(\ref{D}) and (\ref{ly}). Further results of
this type involving volume growth conditions can be found in \cite{S,X}.

In the view of that, we conjecture the following.

\medskip\noindent
\begin{conjecture}[Conjecture 1.]
On an arbitrary complete connected Riemannian manifold $M$, if
\begin{equation*}
\int_{1}^{+\infty }\frac{r^{2q-1}dr}{[\mu (B(o,r))]^{q-1}}=\infty
\end{equation*}%
(in particular, if (\ref{mun}) is satisfied) then there is no positive
solution to (\ref{dx}).
\end{conjecture}
\medskip

One more conjecture is motivated by comparison of {\cite[Theorem 1.1]{GS1}
and Theorem }\ref{last} discussed below.

\medskip\noindent
\begin{conjecture}[Conjecture 2.]
On an arbitrary complete connected Riemannian manifold $M$, if (\ref{log})
is satisfied then, for any $o\in M$,%
\begin{equation*}
\int_{B\left( o,1\right) ^{c}}G\left( x,o\right) ^{\frac{\alpha }{\alpha -2}%
}d\mu \left( x\right) =\infty .
\end{equation*}
\end{conjecture}
\medskip

If this conjecture is true then {\cite[Theorem 1.1]{GS1} follows from
Theorem }\ref{last}.

Our next theorem shows that, for the necessity part of Theorem \ref{main}, it
suffices to have only a lower bound for the Green function.

\begin{theorem}
\label{thm1} Suppose that {(\ref{D}) is satisfied, and let (\ref{int-ineq})
have a positive solution. If the Green function satisfies }the following
lower bound, for some $o\in M$ and all $x\in M$,%
\begin{equation}
G(x,o)\geq C\,\int_{d(x,o)}^{+\infty }\frac{tdt}{\mu (B(o,t))},
\label{cond-l}
\end{equation}%
{\ then }(\ref{cond-int1}) holds for any $r_{0}>0$. If $G$ satisfies the
lower bound
\begin{equation}
G(x,y)\geq C\,\int_{d(x,y)}^{+\infty }\frac{tdt}{\mu (B(o,t))},
\label{cond-lxy}
\end{equation}%
for all $x,y\in M$ then (\ref{cond-int1}) and (\ref{cond-int2}) hold for any
$o\in M$ and $r_{0}>0$.
\end{theorem}

Next, let us consider the following condition on $G$:
\begin{equation}
G(x,y)\approx \tilde{d}(x,y)^{-\gamma },\quad x,y\in M,  \tag{G}
\label{cond-g}
\end{equation}%
where $\tilde{d}$ is some metric on $M$ (not necessarily the geodesic
distance) and $\gamma >0$. The existence of a metric $\tilde{d}$ satisfying (%
\ref{cond-g}) is known to be equivalent to the following inequality, for all
$x,y\in M$:
\begin{equation}
\frac{1}{G(x,y)}\leq \kappa \Big(\frac{1}{G(x,z)}+\frac{1}{G(z,y)}\Big),
\tag{3G}  \label{cond-3g}
\end{equation}%
with some constant $\kappa >0$ ($\kappa $ is called a quasi-metric constant
-- see \cite{FNV}). Indeed, if (\ref{cond-3g}) is satisfied, then $\rho (x,y):=%
\frac{1}{G(x,y)}$ is a quasi-metric, and by the general properties of a
quasi-metric we conclude that $\rho \left( x,y\right) \approx \tilde{d}%
(x,y)^{\gamma }$ for some metric $\tilde{d}$ and $\gamma >0$ (see \cite{Hein}%
), so that (\ref{cond-g}) is satisfied. The converse implication (\ref{cond-g}%
)$\Rightarrow $(\ref{cond-3g}) is obvious (see \cite{GH}).

For example, as we will show below in Lemma \ref{quasi-lem}, estimates (%
\ref{ly}) yield (\ref{cond-3g}).

The next theorem provides necessary and sufficient conditions for the
existence of a positive solution to (\ref{int-ineq}) under hypothesis (\ref%
{cond-g}). Denote by $\tilde{B}(x,r)$ metric balls in the metric $\tilde{d}$.

\begin{theorem}
\label{thm3} Suppose that (\ref{cond-g}) holds for some metric $\tilde{d}$
and $\gamma >0$. Then (\ref{int-ineq}) has a positive solution if and only
if there exist $o\in M$, $r_{0}>0$ and $C>0$ such that
\begin{equation}
\int_{r_{0}}^{+\infty }\frac{\sigma (\tilde{B}(o,t))}{t^{\gamma q+1}}%
dt<\infty ,  \label{cond-1}
\end{equation}%
and
\begin{equation}
\sup_{x\in \tilde{B}(o,r)}\int_{0}^{r}\frac{\sigma (\tilde{B}(x,s))}{%
s^{\gamma +1}}ds\leq C\,r^{\gamma (q-1)},  \label{cond-2}
\end{equation}%
for all $r>r_{0}$.
\end{theorem}

It was proved in \cite[Corollary 2.3]{GS2} that, under hypothesis (\ref{cond-g})
and assuming in addition that
\begin{equation}
\mu (\tilde{B}(o,r)\approx r^{\alpha },\quad r\geq r_{0}>0,  \label{gr}
\end{equation}%
where $\alpha >\gamma $, the inequality
\begin{equation}
\Delta u+u^{q}\leq 0  \label{dx1}
\end{equation}
has no positive solution for any $q\leq \frac{\alpha }{\gamma }$. This
result can be obtained also from Theorem \ref{thm3} as we show in Section %
\ref{examp}. However, the result of \cite[Corollary 2.3]{GS2} remains true even
if (\ref{dx1}) is satisfied in the exterior of a compact in $M$, which is
not covered by Theorem \ref{thm3}.

Assume now that $d\sigma =\Phi (x)d\mu $, where the function $\Phi $ satisfies
the condition
\begin{equation}
\Phi (x)\geq c\,\tilde{d}(x,o)^{m},\quad \mathrm{for}\,\,\tilde{d}(x,o)\geq
r_{0}>0,  \label{gr-phi}
\end{equation}%
with $c>0$ and $m>\gamma -\alpha $. It was proved in \cite[Theorem 2.1]{GS2}
that, under the assumptions (\ref{cond-g}) and (\ref{gr}), (\ref{ineq}) has
no positive solutions for any $q\leq \frac{\alpha +m}{\gamma }.$ This result
can similarly be deduced from Theorem \ref{thm3}.

The following theorem yields the necessary part of Theorem \ref{thm3}
provided $G$ satisfies only the lower bound in (\ref{cond-g}).

\begin{theorem}
\label{thm2}Let (\ref{int-ineq}) have a positive solution. If there is a
metric $\tilde{d}$ on $M$ such that, for some $o\in M$, $\gamma ,r_{0}>0$,
\begin{equation}
G(x,o)\geq c\,\tilde{d}(x,o)^{-\gamma },  \label{cond-gt}
\end{equation}%
for all $x\in M$ such that $\tilde{d}(x,o)\geq r_{0}>0$, then (\ref{cond-1})
is satisfied.

Moreover, if we have, for all $x,y\in M$,%
\begin{equation}
G\left( x,y\right) \geq c\,\tilde{d}(x,y)^{-\gamma },  \label{cond-gxy}
\end{equation}
then (\ref{cond-2}) is satisfied as well.
\end{theorem}

We conclude with more general necessary conditions for the existence of a
positive solution to (\ref{int-ineq}), without imposing any additional a
priori assumptions on the Green function $G$. These conditions
are also sufficient under the assumption (\ref{cond-3g}).

\begin{theorem}
\label{last} If there exists a positive solution to (\ref{int-ineq}), then
for all $o\in M$ and $a>0$, the following conditions hold:
\begin{equation}
\int_{M}\min \Big(G(x,o),a^{-1}\Big)^{q}d\sigma (x)<\infty ,  \label{last-1}
\end{equation}%
and
\begin{equation}
\sup_{x\in M}\int_{\{y\in M:\,G(o,y)>r^{-1}\}}G(x,y)d\sigma (y)\leq
C\,r^{q-1},  \label{last-2}
\end{equation}%
for all $r>a$, where $C$ is a positive constant (that may depend on $q$, $o$
and $a$).

If conditions (\ref{last-1}) and (\ref{last-2}) are satisfied for some $%
o\in M$ and $a>0$ and, in addition, $G$ satisfies (\ref{cond-3g}), then there
exists a positive solution to (\ref{int-ineq}).
\end{theorem}

Under the assumption (\ref{cond-3g}), certain necessary and sufficient
conditions for (\ref{int-ineq}) to have a positive solution were established
in \cite{KV}. In fact, our Theorem \ref{thm3} can be derived
from \cite[Theorem 1.2]{KV}, but we give an independent proof by
deducing it from  Theorem \ref{last}.

The structure of this paper is as follows.

In Section \ref{examp} we give examples of applications of Theorems \ref{main}
and \ref{thm3}.

In Section \ref{prelim} we prove some preparatory results needed for the
proofs of the above theorems. In particular, we prove Proposition \ref%
{prop-m} giving one more necessary and sufficient condition for the existence of
positive solutions of (\ref{int-ineq}) in terms of the Green function, which
however is difficult to verify.

In Section \ref{sec-last} we prove Theorem \ref{last}. This is the most
technical part of the paper. The proof of inequality (\ref{last-1}) is based
on weighted norm inequalities (Lemma \ref{lemma-r}), whereas the proof of (%
\ref{last-2}) uses Moser type iterations of supersolutions for integral
operators (Lemma \ref{lemma-aq}). Let us mention that these highly
non-trivial techniques originate in \cite{KV}.

In Section \ref{Sec23} we prove Theorems \ref{thm2} and \ref{thm3}. For
that, we verify that, if $G$ satisfies hypothesis (\ref{cond-g}), then
conditions (\ref{cond-1}) and (\ref{cond-2}) become equivalent to
conditions (\ref{last-1}) and (\ref{last-2}) of Theorem \ref{last}.

In Section \ref{sec-main} we prove the remaining Theorems \ref{main}, \ref%
{thm1} and Corollary \ref{cor-dx}, also by reducing  to Theorem \ref{last}.

\section{Examples}\label{examp}

\begin{example}{\rm 
Let $M$ be $\mathbb{R}^{n}$ with $n>2$. Then (\ref{D}) and (\ref{ly})
are trivially satisfied. By Corollary \ref{cor-dx}, the inequality
\begin{equation*}
\Delta u+u^{q}\leq 0
\end{equation*}%
has a positive solution if and only if (\ref{cond-int1b}) is satisfied. Since%
\begin{equation*}
\mu \left( B\left( o,r\right) \right) =cr^{n},
\end{equation*}%
we see that (\ref{cond-int1b}) is equivalent to%
\begin{equation*}
\int_{1}^{\infty }\frac{r^{2q-1}}{r^{n\left( q-1\right) }}dr= \int_{1}^{\infty }r^{-\left( n-2\right) q+n-1}dr<\infty ,
\end{equation*}%
that is, to $q>\frac{n}{n-2}$. This result is well known and goes back to
\cite{MP} (see also \cite{SZ}).

Consider now in $\mathbb{R}^{n}$ the inequality
\begin{equation*}
\Delta u+\left\vert x\right\vert ^{m}u^{q}\leq 0.
\end{equation*}%
By Theorem \ref{main} it has a positive solution if and only if
conditions (\ref{cond-int1}), (\ref{cond-int2}) are satisfied with $d\sigma
=|x|^{m}d\mu $. Similarly to the above computation, we obtain that this is
the case if and only if $q>\frac{n+m}{n-2}$ and $m>-2$. The result is also
known and is due to \cite{MP}.
}
\end{example}

\begin{example}{\rm 
Let us recall the following result from \cite[Theorem 2.6]{GS2}. Let $M$
be a Riemannian manifold of bounded geometry such that
\begin{equation*}
G\left( x,y\right) \approx d\left( x,y\right) ^{-\gamma }\ \text{if }d\left(
x,y\right) \geq 1,
\end{equation*}%
and%
\begin{equation*}
\mu \left( B\left( x,r\right) \right) \approx r^{\alpha }\ \ \text{if }r\geq
1,
\end{equation*}%
where $\alpha >\gamma >0.$ Then the inequality%
\begin{equation}
\Delta u+u^{q}\leq 0  \label{ineq-plain}
\end{equation}%
has a positive solution if and only if $q>\frac{\alpha }{\gamma }$.

Let us derive this result from our Theorem \ref{thm3}. In the setting of
\cite[Theorem 2.6]{GS2}, the manifold $M$ satisfies, in fact, the following
conditions (where we assume for simplicity that $n=\dim M>2$): for all $%
x,y\in M$,
\begin{equation}
G(x,y)\approx \left\{
\begin{array}{ll}
d(x,y)^{-\gamma } & \text{if}\,\,d(x,y)\geq 1, \\
&  \\
d(x,y)^{-(n-2)} & \text{if}\,\,\,d(x,y)<1,%
\end{array}%
\right.  \label{G-gamma}
\end{equation}%
and, for all $x\in M$,
\begin{equation}
\mu (B(x,r))\approx \left\{
\begin{array}{ll}
r^{\alpha },\quad \text{if }r\geq 1, &  \\
r^{n},\quad \text{if }r\leq 1. &
\end{array}%
\right.  \label{v-alpha}
\end{equation}
It is easy to see that, for any $\delta _{1},\delta _{2}\in (0,1]$,
\begin{equation*}
\tilde{d}(x,y):=\left\{
\begin{array}{ll}
d(x,y)^{\delta _{1}},\quad d(x,y)>1, &  \\
d(x,y)^{\delta _{2}},\quad d(x,y)\leq 1, &
\end{array}%
\right.
\end{equation*}%
is a new metric on $M$. Choose
\begin{equation*}
\delta _{1}=\frac{\gamma }{\tilde{\gamma}}\text{\ \ and\ \ }\delta _{2}=%
\frac{n-2}{\tilde{\gamma}},
\end{equation*}
where $\tilde{\gamma}$ is large enough to ensure that $\delta _{1},\delta
_{2}\leq 1.$

It follows from (\ref{G-gamma}) that, for all $x,y\in M$,
\begin{equation}
G(x,y)\approx \tilde{d}(x,y)^{-\tilde{\gamma}}.  \label{g-tilde d}
\end{equation}%
Hence, we can apply our Theorem \ref{thm3} with $\sigma =\mu $ in order to
obtain necessary and sufficient condition for the existence of a positive
solution to (\ref{ineq-plain}). Let us estimate the integral in (\ref{cond-1}%
) as follows
\begin{eqnarray*}
\int_{1}^{+\infty }\frac{\sigma (\tilde{B}(o,t))}{t^{\tilde{\gamma}q+1}}
&=&\int_{1}^{+\infty }\frac{\mu (B(o,t^{\frac{1}{\delta _{1}}}))}{t^{\tilde{%
\gamma}q}}\frac{dt}{t} \\
&=&\delta _{1}\int_{1}^{+\infty }\frac{\mu (B(o,r))}{r^{\delta _{1}\tilde{%
\gamma}q}}\frac{dr}{r} \\
&\approx &\int_{1}^{+\infty }\frac{r^{\alpha }}{r^{\gamma q}}\frac{dr}{r},
\end{eqnarray*}%
where we have used the change $t=r^{\delta _{1}}$ and (\ref{v-alpha}).
Clearly, the above integral is finite if and only if
\begin{equation}
q>\frac{\alpha }{\gamma }.  \label{q>}
\end{equation}%
Next, let us estimate the integral in (\ref{cond-2}) by splitting the domain
of integration into $\left[ 0,1\right] $ and $[1,r],$ where $r$ is large
enough. We have
\begin{eqnarray}
\int_{0}^{1}\frac{\sigma (\tilde{B}(x,s))}{s^{\tilde{\gamma}+1}}ds
&=&\int_{0}^{1}\frac{\mu (B(x,s^{\frac{1}{\delta _{2}}}))}{s^{\tilde{\gamma}}%
}\frac{ds}{s}  \notag \\
&=&\delta _{2}\int_{0}^{1}\frac{\mu (B(x,\tau ))}{\tau ^{\delta _{2}\tilde{%
\gamma}}}\frac{d\tau }{\tau }  \notag \\
&\approx &\int_{0}^{1}\frac{\tau ^{n}}{\tau ^{n-2}}\frac{d\tau }{\tau }
\notag \\
&\approx &1,  \label{part1}
\end{eqnarray}%
and
\begin{eqnarray}
\int_{1}^{r}\frac{\sigma (\tilde{B}(x,s))}{s^{\tilde{\gamma}+1}}ds
&=&\int_{1}^{r}\frac{\mu (B(x,s^{\frac{1}{\delta _{1}}}))}{s^{\tilde{\gamma}}%
}\frac{ds}{s}  \notag \\
&=&\delta _{1}\int_{1}^{r^{1/\delta _{1}}}\frac{\mu (B(x,\tau ))}{\tau
^{\delta _{1}\tilde{\gamma}}}\frac{d\tau }{\tau }  \notag \\
&\approx &\int_{1}^{r^{1/\delta _{1}}}\frac{\tau ^{\alpha }}{\tau ^{\gamma }}%
\frac{d\tau }{\tau }  \notag \\
&\approx &r^{\frac{1}{\delta _{1}}\left( \alpha -\gamma \right) }=r^{\tilde{%
\gamma}\frac{\alpha -\gamma }{\gamma }}.  \label{part2}
\end{eqnarray}%
Combining with (\ref{part1}) and (\ref{part2}), we obtain
\begin{equation*}
\int_{0}^{r}\frac{\sigma (\tilde{B}(x,s))}{s^{\tilde{\gamma}+1}}\approx r^{%
\tilde{\gamma}\frac{\alpha -\gamma }{\gamma }}.
\end{equation*}%
Recall that condition (\ref{cond-2}) is
\begin{equation*}
\int_{0}^{r}\frac{\sigma (\tilde{B}(x,s))}{s^{\tilde{\gamma}+1}}\leq Cr^{%
\tilde{\gamma}(q-1)}.
\end{equation*}%
Hence, (\ref{cond-2}) is satisfied if and only if
\begin{equation*}
\frac{\alpha -\gamma }{\gamma }\leq q-1,
\end{equation*}%
which is equivalent to $q\geq \frac{\alpha }{\gamma }$. Combining with (\ref%
{q>}) we recover \cite[Theorem 2.6]{GS2}.
}
\end{example}

\section{Preliminaries}

\label{prelim}In this section we prove some preparatory results necessary
for the proofs of the main theorems. We use some results from \cite{GV1},
\cite{GV2} and \cite{KV}.

For any measure $\omega \in \mathcal{M}^{+}(M)\,$denote by
\begin{equation*}
G\omega (x)=\int_{M}G(x,y)\,d\omega (y)
\end{equation*}%
the Green potential of $\omega .$

Let $o\in M$ and let $a>0$. We set
\begin{equation}
m(x)=m_{a,o}(x)=\min \Big (G(x,o),a^{-1}\Big),  \label{def-m}
\end{equation}%
where sometimes we drop the subscripts $a$ and $o$.

The proof of the following lemma is based on the (local) Harnack inequality
on $M$ (see \cite{G}).

\begin{lemma}
\label{lem-harn} For any $\omega \in \mathcal{M}^{+}(M)$ ($\omega \not=0$),
we have%
\begin{equation}
G\omega (x)\geq C\,m(x)\ \ \text{for all }x\in M,  \label{harn-low}
\end{equation}%
where $C>0$ may depend on $\omega ,o,a$.
\end{lemma}

\begin{proof}
Without loss of generality we may assume that $G\omega \not\equiv +\infty$.
By the lower semicontinuity of $G(x, \cdot)$, it follows that $G \omega$ is
lower semicontinuous, and hence is bounded below by a positive constant on
every compact subset $K$ of $M$ (see also \cite{GS1}).

Without loss of generality, we may assume that $\omega $ is supported in a
fixed compact set $K\subset M$ such that $o\in K$, where $0<\omega
(K)<\infty $.

Let $U$ be a precompact open neighborhood of $K$. To verify (\ref{harn-low}%
), notice first that
\begin{equation*}
c:=\min \,\{G\omega (x):\,x\in \overline{U}\}>0
\end{equation*}%
and, consequently,
\begin{equation*}
G\omega (x)\geq  c \geq c \, a \,  m(x)\ \ \text{for all }x\in U.
\end{equation*}%
For any $x\in M\setminus U$, the function $h(z):=G(x,z)$ is harmonic in $U$.
Hence, by a local Harnack's inequality (see \cite[Theorem 13.10]{G}), we
have $h(z)\geq C_{K,U}\,h(o)$ for all $z\in K$, where $C_{K,U}>0$ is the
local Harnack constant associated with a couple $K,U.$ It follows that, for
all $x\in M\setminus U$,
\begin{align*}
G\omega (x)& =\int_{K}G(x,z)\,d\omega (z) \\
& \geq C_{K,U}\,\int_{K}G(x,o)\,d\omega _{K}(z) \\
& =C_{K,U}\,\omega (K)\,G(x,o)\geq C_{K,U}\,\omega (K)\,m(x).
\end{align*}%
Hence, we obtain (\ref{harn-low}) for all $x\in M.$
\end{proof}

\begin{lemma}
\label{lem-equivalent} Inequality (\ref{int-ineq}) has a positive solution
if and only if the following integral equation
\begin{equation}
u(x)=\int_{M}G(x,y)\,[u(y)]^{q}d\sigma (y)+G\omega ,\quad x\in M,
\label{int}
\end{equation}
has a solution for some measure $\omega \in \mathcal{M}^{+}(M)$ ($\omega
\not=0$), that is, there exists $u>0$ so that
\begin{equation*}
u=G(u^{q}d\sigma )+G\omega .
\end{equation*}
Moreover, $\omega $ can be chosen to be compactly supported in $M$ and with
smooth density with respect to $\sigma $.
\end{lemma}

\begin{proof}
Let $u>0$ be a positive solution of (\ref{int-ineq}). Consider a function $%
v=\varepsilon u$ where $\varepsilon \in \left( 0,1\right) $. We have%
\begin{eqnarray*}
v &=&\varepsilon u\geq \varepsilon G\left( u^{q}d\sigma \right) =\varepsilon
^{1-q}G\left( v^{q}d\sigma \right) \\
&=&G\left( v^{q}d\sigma \right) +\left( \varepsilon ^{1-q}-1\right) G\left(
v^{q}d\sigma \right) \\
&=&G\left( v^{q}d\sigma \right) +\left( \varepsilon -\varepsilon ^{q}\right)
G\left( u^{q}d\sigma \right) .
\end{eqnarray*}%
Since $u$ is positive and lower semi-continuous, it is bounded below by a
positive constant on any compact set. Hence, there exists a non-negative
non-zero function $\varphi \in C_{0}^{\infty }\left( M\right) $ such that $%
u\geq \varphi $ everywhere. It follows that%
\begin{equation}
v\geq G\left( v^{q}d\sigma \right) +h,  \label{vGv}
\end{equation}%
where the function $h=\left( \varepsilon -\varepsilon ^{q}\right) G\left(
\varphi ^{q}d\sigma \right) $ is positive and superharmonic on $M$. By \cite[%
Theorem 5.1]{GV2}, the existence of a positive solution to (\ref{vGv})
implies that
\begin{equation*}
G\left( h^{q}d\sigma \right) \leq \frac{h}{q-1}.
\end{equation*}%
It follows that the function $\tilde{h}:=\delta h$ with $\delta =\left(
\frac{q-1}{q}\right) ^{\frac{q}{q-1}}$ satisfies%
\begin{equation*}
G\left( \tilde{h}^{q}d\sigma \right) \leq \left( 1-\frac{1}{q}\right) ^{q}%
\frac{\tilde{h}}{q-1}.
\end{equation*}%
By \cite[Theorem 3.5]{GV1} (see also \cite{BC}, \cite{KV}), there exists a
positive solution $\tilde{v}$ of the equation
\begin{equation*}
\tilde{v}=G(\tilde{v}^{q}d\sigma )+\tilde{h}.
\end{equation*}%
It follows that $\tilde{v}$ is a positive solution to (\ref{int}) with $%
\omega =\left( \varepsilon -\varepsilon ^{q}\right) \delta \varphi
^{q}d\sigma $.

The converse statement is obvious.
\end{proof}

\begin{lemma}
\label{Lemsmooth} Assume that a measure $\sigma \in \mathcal{M}^{+}(M)$ has a smooth positive density
with respect to $\mu $. If the integral inequality (\ref{int-ineq}) has a
positive solution then the differential inequality (\ref{ineq}) has a positive $%
C^{\infty }$  solution.
\end{lemma}

Of course, conversely, any smooth solution of (\ref{ineq}) also  solves (\ref%
{int-ineq}).

\begin{proof}
By Lemma\textbf{\ }\ref{lem-equivalent}, if (\ref{int-ineq}) has a positive
solution then the integral equation
\begin{equation}
u=G(u^{q}d\sigma )+h  \label{uGh}
\end{equation}%
also has  a positive solution, where $h=G\omega $ as in Lemma \ref%
{lem-equivalent}. Let us mollify $u$ by using a certain heat semigroup in
order to obtain a smooth function. For that, consider the energy form
\begin{equation*}
\mathcal{E}\left( f,f\right) =\int_{M}\left\vert \nabla f\right\vert ^{2}d\mu
\end{equation*}%
in the measure space $(M,\sigma )$. Since $\sigma $ is absolutely continuous
with respect to $\mu $, $\mathcal{E}$ extends to a regular Dirichlet form in
$L^{2}\left( M,\sigma \right) $. The generator of this Dirichlet form is $%
\frac{1}{\Phi }\Delta $ where $\Phi =\frac{d\sigma }{d\mu }$ (see \cite[%
Exercise 3.11]{G}). In particular, the notions of harmonic and superharmonic
functions with respect to $\Delta $ and $\frac{1}{\Phi }\Delta $ are the
same. It is easy to see that the Green functions of $\Delta $ in $\left(
M,\mu \right) $ and $\frac{1}{\Phi }\Delta $ in $\left( M,\sigma \right) $
are the same, so we denote them both by $G\left( x,y\right) $ as before.

Let $\left\{ P_{t}\right\} _{t\geq 0}$ be the heat semigroup of $\mathcal{E}$
in $L^{2}\left( M,\sigma \right) $ and $p_{t}\left( x,y\right) $ be the
corresponding heat kernel, that is, a smooth positive function of $t>0$, $%
x,y\in M$ such that%
\begin{equation*}
P_{t}f\left( x\right) =\int_{M}p_{t}\left( x,y\right) f\left( y\right)
d\sigma \left( y\right) .
\end{equation*}%
For any $t>0$, set%
\begin{equation*}
u_{t}\left( x\right) =P_{t}u\left( x\right) .
\end{equation*}%
Since $u$ is superharmonic, we have $P_{t}u\leq u$. In particular, $u_{t}$
is finite and, hence, $u_{t}\in C^{\infty }\left( M\right) $. Let us prove
that $u_{t}$ satisfies (\ref{ineq}). Using the Green operator%
\begin{equation*}
K=\int_{0}^{\infty }P_{t}dt,
\end{equation*}%
that has in $\left( M,\sigma \right) $ the kernel $G\left( x,y\right) ,$ let
us rewrite (\ref{uGh}) in the form%
\begin{equation*}
u=K(u^{q})+h,
\end{equation*}%
which implies%
\begin{eqnarray*}
u_{t} &=&P_{t}\left( Ku^{q}\right) +P_{t}h \\
&=&K\left( P_{t}u^{q}\right) +P_{t}h \\
&=&K\left( \left( P_{t}u\right) ^{q}\right) +K\left( P_{t}u^{q}-\left(
P_{t}u\right) ^{q}\right) +P_{t}h,
\end{eqnarray*}%
where the operators $K$ and $P_{t}$ commute. Since
\begin{equation*}
\int_{M}p_{t}\left( x,y\right) d\sigma \left( y\right) \leq 1,
\end{equation*}%
we obtain by Jensen's inequality that%
\begin{equation*}
w:=P_{t}u^{q}-\left( P_{t}u\right) ^{q}\geq 0.
\end{equation*}%
Hence, in the identity%
\begin{equation*}
u_{t}=Ku_{t}^{q}+Gw+P_{t}h,
\end{equation*}%
both functions $Gw$ and $P_{t}h$ are superharmonic, and all functions are
smooth. Applying $\frac{1}{\Phi }\Delta $ to both sides of this identity
(see \cite[Lem. 13.1]{G}), we obtain%
\begin{equation*}
\frac{1}{\Phi }\Delta u_{t}\leq -u_{t}^{q}.
\end{equation*}%
Hence, $u_{t}$ solves (\ref{ineq}) for any $t>0$.
\end{proof}

In the next statement, we prove a criterion for solvability of (\ref%
{int-ineq}) and (\ref{ineq}) in terms of the function $m$ defined in (\ref%
{def-m}) for some fixed $o\in M$ and $a>0$.

\begin{proposition}
\label{prop-m} Inequality (\ref{int-ineq}) has a positive solution if and
only if, for some $C>0$,%
\begin{equation}
G[m^{q}d\sigma ](x)\leq C\,m(x),\quad x\in M.  \label{cond-m}
\end{equation}%
If (\ref{cond-m}) is satisfied and $\sigma $ has a smooth positive density
then (\ref{ineq}) has a $C^{\infty }$ solution.
\end{proposition}

\begin{proof}
If (\ref{cond-m}) is satisfied then $u=\varepsilon m$ is a solution of (\ref%
{int-ineq}) for $\varepsilon =C^{-\frac{1}{q-1}}.$ If in addition $\sigma $
has a smooth positive density then, by Lemma \ref{Lemsmooth},  (\ref%
{ineq}) also has a positive solution.

Assume now that (\ref{int-ineq}) has a solution $u>0$. {By Lemma \ref%
{lem-equivalent}, there exists $\omega \in \mathcal{M}^{+}(M)$ ($\omega
\not=0$) such that}
\begin{equation}
u=G(u^{q}d\sigma )+G\omega .  \label{c-om}
\end{equation}%
By Lemma \ref{lem-harn}, we have, for some constant $c>0$, that in $M$%
\begin{equation*}
G\omega \geq cm=:h.
\end{equation*}%
Consequently, $u$ satisfies the inequality
\begin{equation}
u\geq G(u^{q}d\sigma )+h.  \label{c-m}
\end{equation}%
Note that $h=cm$ is obviously superharmonic and, hence, satisfies the
following domination principle:
\begin{equation}
G(fd\sigma )(x)\leq h(x)\ \text{in }\textrm{supp}(f)\ \ \Longrightarrow \ \
G(fd\sigma )(x)\leq h(x)\,\,\text{in\ }M,  \label{D-P}
\end{equation}%
for any bounded measurable function $f\geq 0$ with compact support,\emph{\ }%
such that $G(fd\sigma )$ is bounded on $\textrm{supp}(f)$.

By \cite[Theorem 5.1]{GV2}, the existence of a solution to (\ref{c-m})
implies
\begin{equation*}
G[h^{q}d\sigma ]\leq \frac{1}{q-1}\,h,
\end{equation*}%
which proves (\ref{cond-m}).
\end{proof}

\section{Proof of Theorem \protect\ref{last}}

\label{sec-last}

We will need the following lemma that follows from \cite[Lemma 2.5 and
Remark 2.6]{GV2}. An earlier version of this lemma was obtained in \cite{KV}
for quasi-metric kernels.

\begin{lemma}
\label{lem-r} Let $1<s<\infty $, and let $\sigma \in \mathcal{M}^{+}(M)$ be
a measure such that the Green function $G\left( x,\cdot \right) $ is locally
integrable with respect to $\sigma $ for any $x\in M$. Then, for all $x\in
M, $
\begin{equation}
\left[ G\sigma (x)\right] ^{s}\leq s\,G[(G\sigma )^{s-1}d\sigma ](x).
\label{r-in}
\end{equation}
\end{lemma}

\subsection{Weighted norm inequalities}

\label{wei-in}

The following lemma was obtained earlier in \cite{KV} for quasi-metric
kernels (see also \cite{VW}).

\begin{lemma}
\label{lemma-r} Let $1<q<\infty $, and let $\sigma ,\omega \in \mathcal{M}%
^{+}(M)$. Assume that $G\left( x,\cdot \right) $ is locally integrable with
respect to $\sigma $ and that $G\omega $ is locally bounded. Assume also
that for all $x\in M$
\begin{equation}
G[(G\omega )^{q}d\sigma ](x)\leq c\,G\omega (x).  \label{int-cond}
\end{equation}%
{Then we have
\begin{equation}
||G(fd\sigma )||_{L^{s}(\omega )}\leq C\,||f||_{L^{s}(\sigma )},\quad \text{%
for all}\,\,f\in L^{s}(\sigma ),  \label{super-integral}
\end{equation}%
where }$s=\frac{q}{q-1}$ and $C=sc^{\frac{s-1}{s}}$, {and}
\begin{equation}
||G(gd\omega )||_{L^{q}(\sigma )}\leq C||g||_{L^{q}(\omega )},\quad \text{%
for all}\,\,g\in L^{q}(\omega ).  \label{Ggom}
\end{equation}
\end{lemma}

\begin{proof}
{Let us first prove that, for all $f\in L^{s}(\sigma )$, }%
\begin{equation}
||G(fd\sigma )||_{L^{s}(\nu )}\leq sc\,||f||_{L^{s}(\sigma )},
\label{super-int-a}
\end{equation}%
where $d\nu =(G\omega )^{q}d\sigma $.{\ By a standard approximation
argument, it suffices to prove (\ref{super-int-a}) assuming that $f$ is
non-negative, compactly supported and bounded.}

Using inequality (\ref{r-in}) with $f\,d\sigma $ in place of $\sigma $, we
obtain
\begin{equation}
\lbrack G(f\,d\sigma )]^{s}\leq s\,G\left[ f\,[G(f\,d\sigma
)]^{s-1}\,d\sigma \right] ,  \label{iter-p}
\end{equation}%
whence by Fubini's theorem
\begin{align*}
\int_{M}[G(f\,d\sigma )]^{s}d\omega & \leq s\,\int_{M}G\left[
f\,[G(f\,d\sigma )]^{s-1}\,d\sigma \right] d\omega \\
& =s\,\int_{M}f\,[G(f\,d\sigma )]^{s-1}(G\omega )\,d\sigma .
\end{align*}%
By H\"{o}lder's inequality, the right-hand side is bounded by
\begin{equation*}
\int_{M}f\,[G(f\,d\sigma )]^{s-1}(G\omega )\,d\sigma \leq
||f||_{L^{s}(\sigma )}\left[ \int_{M}[G(f\,d\sigma )]^{s}(G\omega
)^{q}d\sigma \right] ^{\frac{1}{q}}.
\end{equation*}%
Combining the above two estimates, we obtain
\begin{equation}
\int_{M}[G(f\,d\sigma )]^{s}d\omega \leq s\,\,||f||_{L^{s}(\sigma )}\left[
\int_{M}[G(f\,d\sigma )]^{s}d\nu \right] ^{\frac{1}{q}},  \label{est-mu}
\end{equation}%
where $d\nu =(G\omega )^{q}d\sigma $. Using (\ref{iter-p}) and H\"{o}lder's
inequality exactly as above, but with $\nu $ in place of $\omega $, we obtain%
\begin{equation*}
\int_{M}[G(f\,d\sigma )]^{s}d\nu \leq s\,||f||_{L^{s}(\sigma )}\left[
\int_{M}[G(f\,d\sigma )]^{s}(G\nu )^{q}d\sigma \right] ^{\frac{1}{q}}.
\end{equation*}%
By (\ref{int-cond}), we have
\begin{equation*}
G\nu =G[(G\omega )^{q}d\sigma ]\leq c\,G\omega ,
\end{equation*}%
and hence
\begin{equation*}
(G\nu )^{q}d\sigma \leq c^{q}(G\omega )^{q}d\sigma =c^{q}d\nu .
\end{equation*}%
Consequently,
\begin{equation}
\int_{M}[G(f\,d\sigma )]^{s}d\nu \leq sc||f||_{L^{s}(\sigma )}\left[
\int_{M}[G(f\,d\sigma )]^{s}d\nu \right] ^{\frac{1}{q}}.  \label{Gfsi}
\end{equation}%
Let us show that the left hand side here is finite. Indeed, we have%
\begin{eqnarray*}
\int_{M}[G(f\,d\sigma )]^{s}d\nu &=&\int_{M}[G(f\,d\sigma )]^{s}(G\omega
)^{q}d\sigma \\
&\leq &s\int_{M}G\left[ f\,[G(f\,d\sigma )]^{s-1}\,d\sigma \right] \left(
G\omega \right) ^{q}d\sigma \\
&=&s\int_{M}f\,[G(f\,d\sigma )]^{s-1}\,G\left( \left( G\omega \right)
^{q}d\sigma \right) d\sigma \\
&\leq &sc\int_{M}f\,[G(f\,d\sigma )]^{s-1}\,\left( G\omega \right) d\sigma .
\end{eqnarray*}%
Since $f$ is bounded and has a compact support, while $G\omega $ is locally
bounded from below by positive constants, it follows from (\ref{int-cond})
that $G\left( fd\sigma \right) $ is bounded by $\textrm{const}G\omega $. Since
$G\omega $ is locally bounded and the above integral can be reduced to $%
\textrm{supp}f$, we obtain that this integral is finite. Hence, it follows
from (\ref{Gfsi}) that
\begin{equation*}
||G(f\,d\sigma )||_{L^{s}(\nu )}\leq sc\,||f||_{L^{s}(\sigma )},
\end{equation*}%
which proves (\ref{super-int-a}). From (\ref{est-mu}) and (\ref{super-int-a}%
) we obtain%
\begin{equation*}
\int_{M}[G(f\,d\sigma )]^{s}d\omega \leq s^{s}\,c^{s-1}\,||f||_{L^{s}(\sigma
)}^{s},
\end{equation*}%
which proves (\ref{super-integral}). Finally, we prove (\ref{Ggom}) by
duality argument:%
\begin{eqnarray*}
||G(gd\omega )||_{L^{q}(\sigma )} &=&\sup_{f\in L^{s}\left( \sigma \right) }%
\frac{\int_{M}G(gd\omega )fd\sigma }{\left\Vert f\right\Vert _{L^{s}\left(
\sigma \right) }} \\
&=&\sup_{f\in L^{s}\left( \sigma \right) }\frac{\int_{M}gG\left( fd\sigma
\right) d\omega }{\left\Vert f\right\Vert _{L^{s}\left( \sigma \right) }} \\
&\leq &\sup_{f\in L^{s}\left( \sigma \right) }\frac{\left\Vert g\right\Vert
_{L^{q}\left( \omega \right) }\left\Vert G\left( fd\sigma \right)
\right\Vert _{L^{s}\left( \omega \right) }}{\left\Vert f\right\Vert
_{L^{s}\left( \sigma \right) }} \\
&\leq &C\left\Vert g\right\Vert _{L^{q}\left( \omega \right) }.
\end{eqnarray*}
\end{proof}

\subsection{Iterations of supersolutions}

We remark that by Proposition \ref{prop-m}, under the assumptions of Theorem %
\ref{last}, we have that condition (\ref{cond-m}) is necessary and sufficient for the
solvability of (\ref{int-ineq}), that is, for the existence of a non-trivial
superharmonic function $u>0$ such that
\begin{equation}
u(x)\geq G(u^{q}d\sigma )(x),\quad \text{for all}\,\,x\in M.
\label{super-solution}
\end{equation}

For all $x\in M$ and $r>0$ set
\begin{equation}
A(x,r):=\{y\in M:\,\,G(x,y)\geq r^{-1}\},\quad x\in M,\,r>0.  \label{ball-g}
\end{equation}%
We will need the following two lemmas. We start with a preliminary estimate
of $G\sigma _{A}(x)$, where $A=A(o,r)$ and $d\sigma _{A}=\chi _{A}\,d\sigma $.

\begin{lemma}
\label{lemma-ag} Let $1<q<\infty $, and let $\sigma \in \mathcal{M}^{+}(M)$.
Assume that condition (\ref{cond-m}) is satisfied for some $o\in M$ and $a>0$. Then the following estimate holds:
\begin{equation}
G\sigma _{A(o,r)}(x)\leq Cr^{q}\,m(x),\quad \text{for all}\,\,x\in
M,\,\,r\geq a,  \label{est-gp}
\end{equation}%
where the constant $C$ is the same as in (\ref{cond-m}).
\end{lemma}

\begin{proof}
Fix some $r\geq a$. For any $y\in A(o,r)$, we have $G(y,o)\geq r^{-1}$.
Since $a^{-1}\geq r^{-1}$, it follows that
\begin{equation*}
m(y)=\min [G(y,o),a^{-1}]\geq r^{-1}
\end{equation*}%
and, consequently,
\begin{equation*}
G[m^{q}d\sigma _{A(o,r)}](x)=\int_{A(o,r)}G(x,y)\,m(y)^{q}\,d\sigma (y)\geq
r^{-q}\,G\sigma _{A(o,r)}(x).
\end{equation*}%
By (\ref{cond-m}), we have%
\begin{equation}
G[m^{q}d\sigma _{A(o,r)}](x)\leq C\,m(x),\quad \text{for all}\,\,x\in
M,\,\,r\geq a.  \label{cond-m-c}
\end{equation}%
Combining with the previous estimate yields
\begin{equation}
r^{-q}\,G\sigma _{A(o,r)}(x)\leq C\,m(x),\quad \text{for all}\,\,x\in
M,\,\,r\geq a.  \label{cond-m-d}
\end{equation}%
which is equivalent to (\ref{est-gp}).
\end{proof}

The proof of the next lemma is based on Moser type iterations of estimate (%
\ref{cond-m}) and Lemma \ref{lemma-ag}.

\begin{lemma}
\label{lemma-aq} Let $1<q<\infty $, and let $\sigma \in \mathcal{M}^{+}(M)$.
Let (\ref{cond-m}) be satisfied for some $o\in M$ and $a>0$. Then the
following estimate holds
\begin{equation}
G\sigma _{A(o,r)}(x)\leq c\,r^{q-1}  \label{est-ap}
\end{equation}%
for all $x\in M$ and$\,\,r\geq a$, where the constant $c$ may depend on $q$,
$o$ and $a$.
\end{lemma}

\begin{proof}
Fix $r\geq a$. We start with (\ref{cond-m-d}) as our first estimate. Let us
raise (\ref{cond-m-d}) to the power $q$ and apply $G\left( \cdot d\sigma
\right) .$ Using further (\ref{cond-m}), we obtain, for any $x\in M$,%
\begin{equation}
C^{-q}\,\,r^{-q^{2}}G[(G\sigma _{A(o,r)})^{q}d\sigma ](x)\leq G[m^{q}d\sigma
](x)\leq C\,m(x).  \label{j=1}
\end{equation}%
By (\ref{r-in}) with $s=1+q$, we have%
\begin{equation*}
\left( G\sigma _{A\left( o,r\right) }\right) ^{1+q}\left( x\right) \leq
\left( 1+q\right) G\left[ \left( G\sigma _{A\left( o,r\right) }\right)
^{q}d\sigma _{A\left( o,r\right) }\right] \left( x\right) ,
\end{equation*}%
which together with (\ref{j=1}) yields
\begin{equation*}
C^{-q}\,r^{-q^{2}}\,(1+q)^{-1}(G\sigma _{A(o,r)})^{1+q}\,\leq C\,m(x).
\end{equation*}%
Raising again to the power $q$, applying $G\left( \cdot d\sigma \right) $
and using (\ref{cond-m}), we obtain
\begin{equation*}
C^{-q^{2}}\,r^{-q^{3}}\,(1+q)^{-q}G[(G\sigma _{A(o,r)})^{q(1+q)}d\sigma
_{A(o,r)}](x)\leq C^{1+q}\,m(x).
\end{equation*}%
By (\ref{r-in}) with $s=1+q+q^{2}=1+q(1+q)$, have%
\begin{equation*}
\left( G\sigma _{A\left( o,r\right) }\right) ^{1+q+q^{2}}\left( x\right)
\leq \left( 1+q+q^{2}\right) G\left[ \left( G\sigma _{A\left( o,r\right)
}\right) ^{q\left( 1+q\right) }d\sigma _{A\left( o,r\right) }\right] \left(
x\right) ,
\end{equation*}%
whence we deduce our third iteration
\begin{equation*}
C^{-q^{2}}\,r^{-q^{3}}\,(1+q)^{-q}(1+q+q^{2})^{-1}(G\sigma
_{A(o,r)})^{1+q+q^{2}}(x)\leq C^{1+q}\,m(x).
\end{equation*}%
Iterating this process further, we obtain, for our $j$-th iteration, as in
\cite[Corollary 2.8]{GV2}, that
\begin{equation}
C^{-q^{j-1}}r^{-q^{j}}\,c(j,q)(G\sigma _{A(o,r)})^{1+q+q^{2}+\cdots
+q^{j-1}}(x)\leq C^{1+q+...+q^{j-1}}\,m(x),  \label{mj}
\end{equation}%
where
\begin{equation*}
c(j,q)=\prod_{k=1}^{j-1}(1+q+q^{2}+\cdots +q^{k})^{-q^{j-1-k}}.
\end{equation*}%
Now we raise both sides of (\ref{mj}) to the power $q^{-j}$, and let $%
j\rightarrow \infty $. Note that, as in the proof of \cite[Theorem 3.8]{KV},
the infinite product
\begin{align*}
& \prod_{k=1}^{\infty }(1+q+q^{2}+\cdots +q^{k})^{q^{-1-k}} \\
=& \prod_{k=1}^{\infty }q^{kq^{-1-k}}\,\prod_{k=1}^{\infty
}(1+q^{-1}+q^{-2}+\cdots +q^{-k})^{q^{-1-k}} \\
\leq & \prod_{k=1}^{\infty }q^{kq^{-1-k}}\,\prod_{k=1}^{\infty }\left( \frac{%
q}{q-1}\right) ^{q^{-1-k}} \\
=& q^{(q-1)^{-2}}\left( \frac{q}{q-1}\right) ^{\frac{1}{q(q-1)}}
\end{align*}%
is convergent. Hence,
\begin{equation*}
c(q)=\lim_{j\rightarrow \infty }c(j,q)^{q^{-j}}=\prod_{k=1}^{\infty
}(1+q+q^{2}+\cdots +q^{k})^{-q^{-1-k}}>0,
\end{equation*}%
from which we obtain
\begin{equation*}
C^{-\frac{1}{q}}\,\,r^{-1}\,c(q)\,(G\sigma _{A(o,r)})^{\frac{1}{q-1}}(x)\leq
C^{\frac{1}{q-1}},
\end{equation*}%
which completes the proof of (\ref{est-ap}).
\end{proof}

\subsection{Completion of proof of Theorem \protect\ref{last}}

\begin{proof}[Proof of necessity]
Assume that (\ref{int-ineq}) has a positive solution. Fix $o\in M$, $a>0$
and define $m$ by (\ref{def-m}), that is,%
\begin{equation*}
m(x):=\min \left( G(x,o),a^{-1}\right) .
\end{equation*}

By Proposition \ref{prop-m}, we have that (\ref{cond-m}) is satisfied. Setting
\begin{equation*}
d\omega =m^{q}d\sigma ,
\end{equation*}%
we obtain%
\begin{equation*}
G\omega \leq Cm,
\end{equation*}%
which, in particular, implies that $G\omega $ is bounded. Raising this
inequality to the power $q$ and integrating against $d\sigma $, we obtain%
\begin{equation*}
G\left( \left( G\omega \right) ^{q}d\sigma \right) \leq c \, G\omega ,
\end{equation*}%
with $c=C^{q}$, which coincides with the hypothesis (\ref{int-cond}) of
Lemma \ref{lemma-r}. By this lemma, we have (\ref{Ggom}), that is, for all $%
g\in L^{q}\left( \omega \right) $,
\begin{equation}
||G(gd\omega )||_{L^{q}(\sigma )}\leq C\,||g||_{L^{q}(\omega )}.
\label{dual-int}
\end{equation}%
Let $K$ be a compact subset of $M$ such that $\omega \left( K\right) >0$.
Notice that $\omega \left( K\right) <\infty $, since $\sigma $ is a Radon
measure and $m$ is bounded. Letting $g=\chi _{K}$ in (\ref{dual-int})  and
observing that by Lemma \ref{lem-harn}%
\begin{equation*}
m\leq CG\left( \chi _{K}d\omega \right) =CG\left( gd\omega \right) ,
\end{equation*}%
we obtain
\begin{equation}
\left\Vert m\right\Vert _{L^{q}\left( \sigma \right) }<\infty ,
\label{cond-m-ql}
\end{equation}%
which proves condition (\ref{last-1}) of Theorem \ref{last}.


In order to prove (\ref{last-2}), observe that by Proposition \ref{prop-m}
we have (\ref{cond-m}). Hence, the hypotheses of Lemma \ref{lemma-aq} are
satisfied, and we conclude by this lemma that (\ref{est-ap}) hold, which
coincides with (\ref{last-2}). This completes the proof of the necessity
part of Theorem \ref{last}.
\end{proof}

\begin{proof}[Proof of sufficiency]
Let us prove that under hypotheses (\ref{last-1}), (\ref{last-2}) and (%
\ref{cond-3g}), inequality (\ref{int-ineq}) has a positive solution. By
Proposition \ref{prop-m}, it suffices to verify (\ref{cond-m}), that is,%
\begin{equation*}
G\left( m^{q}d\sigma \right) \left( x\right) \leq Cm\left( x\right) ,
\end{equation*}%
for all $x\in M$, where%
\begin{equation*}
m(x)=\min \left( G\left( o, x \right), a^{-1}\right) .
\end{equation*}%
Hence, it suffices to verify that, for all $x\in M$%
\begin{equation}
G(m^{q}d\sigma )(x)\leq Ca^{-1}  \label{Ga}
\end{equation}%
and%
\begin{equation}
G(m^{q}d\sigma )(x)\leq CG\left( o,x\right) .  \label{GG}
\end{equation}%
Using integration with respect to the level sets of $m$ and noticing that $%
0\leq m\leq a^{-1}$, we obtain
\begin{eqnarray*}
G(m^{q}d\sigma )(x) &=&\int_{M}G\left( x,y\right) m^{q}\left( y\right)
d\sigma \left( y\right) \\
&=&q\int_{0}^{a^{-1}}\left( \int_{\left\{ y\in M: \, m\left( y\right) >t\right\}
}G\left( x,y\right) d\sigma \left( y\right) \right) t^{q-1}dt \\
&\leq &q\int_{0}^{a^{-1}}\left( \int_{\left\{ y\in M: \, G\left( o,y\right)
>t\right\} }G\left( x,y\right) d\sigma \left( y\right) \right) t^{q-1}dt \\
&=&q\int_{a}^{\infty }\left( \int_{\left\{ y\in M: \, G\left( o,y\right)
>r^{-1}\right\} }G\left( x,y\right) d\sigma \left( y\right) \right)
r^{-q-1}dr \\
&\leq &q  C  \int_{a}^{\infty }r^{q-1}r^{-q-1}dr=q C  a^{-1},
\end{eqnarray*}%
where in the last line we used (\ref{last-2}). Hence, (\ref{Ga}) is
proved.

In order to prove (\ref{GG}), let us set
\begin{equation*}
R=G(x,o),
\end{equation*}%
 where in view of (\ref{Ga})  we may assume that $R<(2 \kappa a)^{-1}$,
 and split the domain of integration in $G(m^{q}d\sigma )$ into two parts:
\begin{equation*}
G(x,y)\leq 2\kappa R\ \ \text{and\ \ }G(x,y)>2\kappa R,
\end{equation*}%
where $\kappa $ is the constant from (\ref{cond-3g}). In the first part, we
have by (\ref{last-1})
\begin{eqnarray*}
\int_{\left\{ y\in M:\,G(x,y)\leq 2\kappa R\right\} }G(x,y)m^{q}(y)d\sigma
(y) &\leq &2\kappa \,R\,\int_{M}m^{q}(y)d\sigma (y) \\
&=&CR=CG\left( x,o\right) .
\end{eqnarray*}%
In the second part, we have $G(x,y)>2\kappa R$ and hence,
\begin{equation*}
\frac{1}{G\left( x,y\right) }<\frac{1}{2\kappa G\left( x,o\right)},
\end{equation*}%
which implies by (\ref{cond-3g})
\begin{eqnarray*}
\frac{1}{G\left( x,o\right) } &\leq &\kappa \left( \frac{1}{G\left(
y,o\right) }+\frac{1}{G\left( x,y\right) }\right) \\
&\leq &\kappa \left( \frac{1}{G\left( y,o\right) }+\frac{1}{2\kappa G\left(
x,o\right) }\right) \\
&=&\frac{\kappa }{G\left( y,o\right) }+\frac{1}{2G\left( x,o\right) }.
\end{eqnarray*}%
It follows that%
\begin{equation*}
\frac{1}{2G\left( x,o\right) }\leq \frac{\kappa }{G\left( y,o\right) },
\end{equation*}%
and hence $G(y,o)\leq 2\kappa G(x,o)$. Consequently, we obtain
\begin{equation*}
m\left( y\right) \leq 2\kappa G(x,o),
\end{equation*}%
and by (\ref{last-2}) with $r=(2\kappa R)^{-1}>a$,
\begin{equation*}
\int_{\left\{ y\in M:\,G(y,o)>2\kappa R\right\} }G(x,y)\,m(y)^{q}d\sigma
(y)\leq C\,G(x,o)\,^{q}\left( 2\kappa R\right) ^{-\left( q-1\right)
}=CG\left( x,o\right) .
\end{equation*}%
Combining with the previous estimate, we obtain (\ref{cond-m}), thus
finishing the proof.
\end{proof}

%
%

\section{Proofs of Theorems \protect\ref{thm2} and \protect\ref{thm3}}

\label{Sec23}We prove here Theorems \ref{thm2} and \ref{thm3} using our
Theorem \ref{last}. Hence, in the proof of the necessary conditions in
Theorems \ref{thm2} and \ref{thm3} we can assume that conditions (\ref%
{last-1}) and (\ref{last-2}) are satisfied, for any $a>0$.

\begin{proof}[Proof of Theorem \protect\ref{thm2}]
Using (\ref{cond-gt}), (\ref{last-1}) with large enough $a$ and integration
with respect to the level sets, obtain
\begin{align}
& \infty >\int_{M}\min \left( G\left( x,o\right) ,a^{-1}\right) ^{q}d\sigma
\left( x\right)  \notag \\
& \geq \int_{\left\{ x\in M: \, \tilde{d}(x,o)\geq r_{0}\right\} }\min \left(
G\left( x,o\right) ,a^{-1}\right) ^{q}d\sigma \left( x\right)  \label{2} \\
& \geq c\int_{\left\{ x\in M: \, \tilde{d}(x,o)\geq r_{0}\right\} }\left[ \min
\left( \tilde{d}(x,o)^{-\gamma },a^{-1}\right) \right] ^{q}d\sigma (x)
\label{3} \\
& =cq\,\int_{0}^{\min \left( r_{0}^{-\gamma },a^{-1}\right) }\sigma \left(
\{x\in M: \,\min \left( (\tilde{d}(x,o))^{-\gamma },a^{-1}\right) >s\}\right)
\,s^{q-1}\,ds  \notag \\
& =cq\,\int_{0}^{a^{-1}}\sigma \left( \{x\in M: \,\tilde{d}(x,o))^{-\gamma
}>s\}\right) \,s^{q-1}\,ds  \notag \\
& =cq\,\int_{0}^{a^{-1}}\sigma \left( \tilde{B}\left( o,s^{-1/\gamma
}\right) \right) \,s^{q-1}\,ds  \notag \\
& =c\gamma q\,\int_{a^{\gamma }}^{\infty }\sigma (\tilde{B}(o,r))\frac{dr}{%
r^{\gamma \,q+1}},  \notag
\end{align}%
whence (\ref{cond-1}) follows.

Assuming that (\ref{cond-gxy}) is satisfied, let us deduce (\ref{cond-2}).
We have, for any $r>0,$
\begin{equation*}
G\left( o,y\right) >r^{-1}\Leftarrow \tilde{d}(o,y)<\left( cr\right)
^{1/\gamma }=:\rho ,
\end{equation*}%
so that%
\begin{equation*}
\left\{ y\in M:G\left( o,y\right) >r^{-1}\right\} \supset \tilde{B}(o,\rho ).
\end{equation*}%
Applying (\ref{cond-gxy}) again, we obtain%
\begin{equation*}
\int_{\left\{ y\in M: \, G\left( o,y\right) >r^{-1}\right\} }G\left( x,y\right)
d\sigma \left( y\right) \geq c\int_{\tilde{B}(o,\rho )}\tilde{d}%
(x,y)^{-\gamma }d\sigma (y),
\end{equation*}%
which together with (\ref{last-2}) yields%
\begin{equation*}
\int_{\tilde{B}(o,\rho )}\tilde{d}(x,y)^{-\gamma }d\sigma (y)\leq C\rho
^{\gamma (q-1)},
\end{equation*}%
for all $x\in M$ and $\rho >\left( ca\right) ^{1/\gamma }$. Using
integration with respect to the level sets of $\tilde{d}(x,\cdot )$, we
obtain
\begin{equation*}
\int_{\tilde{B}(o,\rho )}\tilde{d}(x,y)^{-\gamma }d\sigma (y)=\gamma
\int_{0}^{\infty }\sigma \left( \tilde{B}(o,\rho )\cap \tilde{B}(x,s)\right)
\,s^{-\gamma -1}ds.
\end{equation*}%
whence
\begin{equation}
\int_{0}^{\infty }\sigma \left( \tilde{B}(o,\rho )\cap \tilde{B}(x,s)\right)
\,s^{-\gamma -1}ds\leq C\rho ^{\gamma (q-1)},  \label{last2eq}
\end{equation}%
for all $x\in M$ and $\rho >\left( ca\right) ^{1/\gamma }$.

If $x\in \tilde{B}(o,\frac{\rho }{2})$ and $0<s\leq \frac{\rho }{2}$, then $%
\tilde{B}(x,s)\subset \tilde{B}(o,\rho )$. Hence, we obtain, for all $x\in
\tilde{B}(o,\frac{\rho }{2})$,
\begin{equation*}
\int_{0}^{\frac{\rho }{2}}\sigma (\tilde{B}(x,s))\,s^{-\gamma -1}ds\leq
C\,\rho ^{\gamma (q-1)}.
\end{equation*}%
Denoting $r=\frac{\rho }{2}$, we obtain the condition (\ref{cond-2}).
\end{proof}

\begin{proof}[Proof of Theorem \protect\ref{thm3}]
As was mentioned above, condition (\ref{cond-g}) is equivalent to (\ref%
{cond-3g}), that is, $G$ is a quasi-metric kernel. Therefore, Theorem \ref{thm3}
can be deduced from \cite[Theorem 4.10]{KV}. However, we give here an independent
proof.

The necessity of conditions (\ref{cond-1}) and (\ref{cond-2}) follows from
Theorem \ref{thm2}.

We will prove the sufficiency of conditions (\ref{cond-1}) and (\ref{cond-2}%
) by showing that they imply conditions (\ref{last-1}) and (\ref{last-2}%
), respectively. Consequently, the existence of a solution of (\ref{int-ineq}) follows by the second
part of Theorem \ref{last}.

In the proof of Theorem \ref{thm2} we have shown that (\ref{cond-gt}), (\ref%
{last-1}) implies (\ref{cond-1}). The same argument shows that, if (\ref%
{cond-g}) holds, then (\ref{cond-1}) implies (\ref{last-1}). Indeed, in (\ref%
{3}) we have $\approx $ instead of $\geq $, so that (\ref{cond-1}) yields%
\begin{equation*}
\int_{\left\{ x\in M: \, \tilde{d}(x,o)\geq r_{0}\right\} }\min \left( G\left(
x,o\right) ,a^{-1}\right) ^{q}d\sigma \left( x\right) <\infty .
\end{equation*}%
Since%
\begin{equation*}
\int_{\left\{ x\in M: \, \tilde{d}(x,o)<r_{0}\right\} }\min \left( G\left(
x,o\right) ,a^{-1}\right) ^{q}d\sigma \left( x\right) \leq a^{-q}\sigma (%
\widetilde{B}\left( o,r_{0}\right) )<\infty ,
\end{equation*}%
we obtain (\ref{last-1}).

Let us now obtain (\ref{last-2}). It follows from (\ref{cond-1}) and the
monotonicity of $\sigma (\tilde{B}(o,t))$ in $t$, that
\begin{equation}
\sigma (\tilde{B}(o,t))\leq Ct^{\gamma q},  \label{es-metball}
\end{equation}%
for all $t>r_{0}.$ Under the hypotheses (\ref{cond-g}), that is,
\begin{equation*}
G(x,y)\approx \tilde{d}(x,y)^{-\gamma },\quad \text{for all }x,y\in M.
\end{equation*}%
the condition (\ref{last-2}), that is,
\begin{equation*}
\sup_{x\in M}\int_{\{y: \, G(o,y)>r^{-1}\}}G(x,y)d\sigma (y)\leq Cr^{q-1},\quad
\text{for all }r>r_{0},
\end{equation*}%
where $r_{0}>0$, is clearly equivalent to
\begin{equation*}
\sup_{x\in M}\int_{\tilde{B}(o,t)}\tilde{d}(x,y)^{-\gamma }d\sigma (y)\leq
Ct^{\gamma (q-1)}\ \ \text{for all }t>t_{0},
\end{equation*}%
where $t_{0}>0$. The latter is in turn equivalent to%
\begin{equation*}
\int_{0}^{\infty }\frac{\sigma (\tilde{B}(x,s)\cap \tilde{B}(o,t))}{%
s^{\gamma +1}}ds\leq Ct^{\gamma (q-1)}
\end{equation*}%
for all $x\in M$ and $t>t_{0}$. Note first that by (\ref{es-metball})%
\begin{equation*}
\int_{t}^{\infty }\frac{\sigma (\tilde{B}(x,s)\cap \tilde{B}(o,t))}{%
s^{\gamma +1}}ds\leq \int_{t}^{\infty }\frac{Ct^{\gamma q}}{s^{\gamma +1}}%
ds=Ct^{\gamma \left( q-1\right) }.
\end{equation*}%
To estimate a similar integral from $0$ to $t$, observe first that the
intersection%
\begin{equation*}
\tilde{B}(x,s)\cap \tilde{B}(o,t)
\end{equation*}%
is empty if%
\begin{equation*}
\tilde{d}\left( x,o\right) \geq s+t,
\end{equation*}%
In particular, if $\tilde{d}\left( x,o\right) \geq 2t$, then%
\begin{equation*}
\int_{0}^{t}\frac{\sigma (\tilde{B}(x,s)\cap \tilde{B}(o,t))}{s^{\gamma +1}}%
ds=0.
\end{equation*}%
Assume now that $\tilde{d}\left( x,o\right) <2t$. Then $x\in \tilde{B}(o,2t)$
and we obtain by (\ref{cond-2}) that%
\begin{equation*}
\int_{0}^{t}\frac{\sigma (\tilde{B}(x,s)\cap \tilde{B}(o,t))}{s^{\gamma +1}}%
ds\leq \int_{0}^{2t}\frac{\sigma (\tilde{B}(x,s))}{s^{\gamma +1}}ds\leq
Ct^{\gamma \left( q-1\right) },
\end{equation*}%
which was to be proved.
\end{proof}

\section{Proofs of Theorems \protect\ref{thm1}, \protect\ref{main} and
Corollary \protect\ref{cor-dx}}

\label{sec-main}\label{sec-thms}Fix\textbf{\ }$o\in M$ and define for any $%
\rho >0$
\begin{equation}
R(\rho ):=\int_{\rho }^{+\infty }\frac{tdt}{\mu (B(o,t))}.  \label{R}
\end{equation}%
Notice that $R(\rho )$ is a decreasing function of $\rho $, and by the
doubling property (\ref{D}),%
\begin{equation}
R(\rho )\leq CR(2\rho ),\quad \rho >0.  \label{R1}
\end{equation}%
Indeed, letting $t=2s$ in (\ref{R}), we obtain
\begin{equation*}
R(2\rho )=\int_{2\rho }^{+\infty }\frac{tdt}{\mu (B(o,t))}=4\int_{\rho
}^{+\infty }\frac{sds}{\mu (B(o,2s))}\geq cR(\rho ).
\end{equation*}

\begin{lemma}
\label{quasi-lem} Suppose $\beta >0$, and $R(\rho )$ ($\rho >0$) satisfies (%
\ref{R}). We set
\begin{equation*}
\tilde{d}(x,y)=1/R(d(x,y)),\quad x,y\in M.
\end{equation*}%
If the doubling property (\ref{D}) holds, then $\tilde{d}$ satisfies the
quasi-triangle inequality
\begin{equation*}
\tilde{d}(x,z)\leq \kappa \,\left[ \tilde{d}(x,y)+\tilde{d}(y,z)\right]
,\quad x,y,z\in M,
\end{equation*}%
with some $\kappa >0$.
\end{lemma}

\begin{proof}
Clearly, it suffices to prove
\begin{equation}
\min \left[ R(d(x,y)),R(d(y,z))\right] \leq \kappa \,R(d(x,z)).  \label{max}
\end{equation}%
Since $d$ is a metric, for every triple $x,y,z\in M$, we have that either $%
d(x,y)\geq \frac{1}{2}d(x,z)$, or $d(y,z)\geq \frac{1}{2}d(x,z)$.

If $d(x,y)\geq \frac{1}{2}d(x,z)$, then by (\ref{R1}),
\begin{equation*}
R(d(x,y))\leq R\left( \frac{1}{2}d(x,z)\right) \leq C\,R\left( d(x,z)\right)
.
\end{equation*}%
If $d(y,z)\geq \frac{1}{2}d(x,z)$, then similarly
\begin{equation*}
R(d(y,z))\leq R\left( \frac{1}{2}d(x,z)\right) \leq C\,R\left( d(x,z)\right)
,
\end{equation*}%
which finishes the proof of (\ref{max}).
\end{proof}

\begin{proof}[Proof of Theorem \protect\ref{thm1}]
By Theorem \ref{last}, the existence a positive solution (\ref{int-ineq})
implies (\ref{last-1}) and (\ref{last-2}) for any $a>0$. Using condition (%
\ref{cond-l}), we can replace $G(x,o)$ in (\ref{last-1}) by $R(d(x,o))$,
where the function $R\left( \rho \right) $ is defined by (\ref{R}), thus
obtaining%
\begin{equation}
\int_{M}\min \left( R(d(x,o)),a^{-1}\right) ^{q}d\sigma <\infty .
\label{int-R}
\end{equation}%
Integration in level sets of $\min \left( R(d(x,o)),a^{-1}\right) $ yields%
\begin{align}
& \int_{M}\min \left( R(d(x,o)),a^{-1}\right) ^{q}d\sigma  \notag \\
& =q\,\int_{0}^{\infty }\sigma \left( \{x\in M:\,\min \left(
R(d(x,o)),a^{-1}\right) >s\}\right) \,s^{q-1}\,ds  \notag \\
& =q\,\int_{0}^{a^{-1}}\sigma \left( \{x\in M:\,R(d(x,o))>s\}\right)
\,s^{q-1}\,ds.  \label{3a}
\end{align}%
Making here a change $s=R\left( r\right) $, observing that
\begin{equation*}
\{x\in M:\,R(d(x,o))>s\}=B(o,r),
\end{equation*}%
and setting $a^{-1}=R(r_{0})$ we obtain that
\begin{equation}
\int_{M}\min \left( R(d(x,o)),a^{-1}\right) ^{q}d\sigma
=q\,\int_{r_{0}}^{\infty }\left[ \int_{r}^{\infty }\frac{tdt}{\mu (B(o,t))}%
\right] ^{q-1}\frac{\sigma (B(o,r))}{\mu (B(o,r))}r\,dr,  \label{Rsigma}
\end{equation}%
which together with (\ref{int-R}) finishes the proof of (\ref{cond-int1}).

{Let us now deduce (\ref{cond-int2}), assuming (\ref{D}) and (\ref{cond-lxy}%
). By (\ref{cond-lxy}), we have, for any $r>0,$
\begin{equation*}
G\left( o,y\right) >r^{-1}\Leftarrow R\left( d\left( o,y\right) \right)
>\left( cr\right) ^{-1}.
\end{equation*}%
If }$r>r_{0}$ for some large $r_{0}$ then the equation%
\begin{equation*}
\left( cr\right) ^{-1}=\int_{\rho }^{\infty }\frac{tdt}{\mu (B(o,t))}%
=R\left( \rho \right)
\end{equation*}%
has a unique positive solution $\rho =\rho \left( r\right) $. {Hence,}%
\begin{equation*}
G\left( o,y\right) >r^{-1}\Leftarrow R\left( d\left( o,y\right) \right)
>R\left( \rho \right) \Leftrightarrow d\left( o,y\right) <\rho ,
\end{equation*}%
so that{%
\begin{equation*}
\left\{ y\in M:G\left( o,y\right) >r^{-1}\right\} \supset B(o,\rho ).
\end{equation*}%
By (\ref{cond-lxy}), we obtain%
\begin{equation}
\int_{\left\{ y\in M:G\left( o,y\right) >r^{-1}\right\} }G\left( x,y\right)
d\sigma \left( y\right) \geq c\int_{B(o,\rho )}R\left( d\left( x,y\right)
\right) d\sigma (y),  \label{G>R}
\end{equation}%
On the other hand, by Fubini's theorem, we have }%
\begin{eqnarray}
\int_{B(o,\rho )}R\left( d\left( x,y\right) \right) d\sigma (y)
&=&\int_{B(o,\rho )}\int_{d(x,y)}^{\infty }\frac{tdt}{\mu (B(x,t))}d\sigma
(y)  \notag \\
&=&\int_{0}^{\rho }\frac{\sigma (B(o,\rho )\cap B(x,t))}{\mu (B(x,t))}tdt.
\label{R2}
\end{eqnarray}%
Hence, {\ (\ref{last-2}) yields, for all }$x\in M$ and $r>a$,{\
\begin{equation*}
\int_{0}^{\rho }\frac{\sigma (B(o,\rho )\cap B(x,t))}{\mu (B(x,t))}tdt\leq
Cr^{q-1}.
\end{equation*}%
}If $x\in B(o,\frac{\rho }{2})$ and $0<s\leq \frac{\rho }{2}$, then $%
B(x,s)\subset B(o,\rho )$. Hence, we obtain, for all $x\in B(o,\frac{\rho }{2%
})$,
\begin{equation*}
\int_{0}^{\frac{\rho }{2}}\frac{\sigma (B(x,t))}{\mu (B(x,t))}tdt\leq
Cr^{q-1}=CR\left( \rho \right) ^{-\left( q-1\right) }.
\end{equation*}%
Using (\ref{R1}), we conclude that%
\begin{equation*}
\int_{0}^{\frac{\rho }{2}}\frac{\sigma (B(x,t))}{\mu (B(x,t))}tdt\leq
CR\left( \rho /2\right) ^{-\left( q-1\right) }.
\end{equation*}%
Renaming $\rho /2$ by $r$ we obtain (\ref{cond-int2}).
\end{proof}



\begin{proof}[Proof of Theorem \protect\ref{main}]
The necessity of conditions (\ref{cond-int1}) and (\ref{cond-int2}) follows
from Theorem \ref{thm1}.

Let us prove that (\ref{cond-int1}) and (\ref{cond-int2}) are sufficient for
the existence of a positive solution of (\ref{int-ineq}). It follows from (%
\ref{ly}) and Lemma \ref{quasi-lem} that $G$ satisfies (\ref{cond-3g}).
Hence, by Theorem \ref{last}, it suffices to verify the conditions (\ref%
{last-1}) and (\ref{last-2}).

Indeed, by (\ref{cond-int1}), the right hand side of (\ref{Rsigma}) is
finite, which together with (\ref{ly}) implies (\ref{last-1}).

Let us now verify (\ref{last-2}), assuming that (\ref{ly}), (\ref{D}),(\ref%
{cond-int1}) and (\ref{cond-int2}) are satisfied. Using (\ref{ly}) and
arguing as in the proof of Theorem \ref{thm1}, we obtain, for all $x\in M$, $%
r>0$ and $\rho $ such that $\left( cr\right) ^{-1}=R\left( \rho \right) $:
\begin{eqnarray*}
\int_{\{y\in M:G(o,y)>r^{-1}\}}G(x,y)d\sigma (y) &\leq &C\int_{B(o,\rho
)}R\left( d\left( x,y\right) \right) d\sigma (y) \\
&=&\int_{0}^{\rho }\frac{\sigma (B(o,\rho )\cap B(x,t))}{\mu (B(x,t))}tdt \\
&\leq &\int_{0}^{\rho }\frac{\sigma (B(x,t))}{\mu (B(x,t))}tdt,
\end{eqnarray*}%
Estimating the right hand side by (\ref{cond-int2}), we obtain
\begin{equation*}
\int_{\{y\in M:G(o,y)>r^{-1}\}}G(x,y)d\sigma (y)\leq C\left( \int_{\rho
}^{\infty }\frac{tdt}{\mu (B(o,t))}\right) ^{1-q}=CR\left( \rho \right)
^{1-q}=Cr^{q-1},
\end{equation*}%
which is exactly (\ref{last-2}).
\end{proof}

For the proof of Corollary \ref{cor-dx} we will need the following lemma.

\begin{lemma}
\label{prop-s} Let $s\in (0,1)$, an let $\phi :\,(0,+\infty )\rightarrow
(0,+\infty )$ be a non-increasing function. Then there exists a positive
constant $C=C(s)$ such that, for all $r>0$,
\begin{equation}
\left( \int_{r}^{\infty }\phi (t)\,t\,dt\right) ^{s}\leq C\int_{r}^{\infty
}\phi (t)^{s}\,t^{2s-1}\,dt+Cr^{2s}\,\phi (r)^{s}.  \label{est-decr}
\end{equation}
\end{lemma}

\begin{proof}
We have
\begin{align*}
\left( \int_{r}^{\infty }\phi (t)\,t\,dt\right) ^{s}& =s\,\int_{r}^{\infty
}\left( \int_{r}^{t}\phi (\tau )\,\tau \,d\tau \right) ^{s-1}\phi (t)\,t\,dt
\\
& \leq s\,\int_{r}^{\infty }\left( \int_{r}^{t}\,\tau \,d\tau \right)
^{s-1}\phi (t)^{s}\,t\,dt \\
& =s\,2^{1-s}\,\int_{r}^{\infty }\left( t^{2}-r^{2}\right) ^{s-1}\phi
(t)^{s}\,t\,dt \\
& =s\,2^{1-s}\,(I_{1}+I_{2}),
\end{align*}%
where
\begin{equation*}
I_{1}=\int_{2r}^{\infty }\left( t^{2}-r^{2}\right) ^{s-1}\phi
(t)^{s}\,t\,dt,\quad I_{2}=\int_{r}^{2r}\left( t^{2}-r^{2}\right) ^{s-1}\phi
(t)^{s}t\,dt.
\end{equation*}%
Clearly, for $t>2r$,
\begin{equation*}
\left( t^{2}-r^{2}\right) ^{s-1}\leq \left( \frac{4}{3}\right)
^{1-s}\,t^{2(s-1)},
\end{equation*}%
whence
\begin{equation*}
I_{1}\leq \left( \frac{4}{3}\right) ^{1-s}\,\int_{2r}^{\infty }\phi
(t)^{s}\,t^{2s-1}\,dt.
\end{equation*}%
On the other hand, the change $\xi =t^{2}-r^{2}$ yields
\begin{equation*}
\int_{r}^{2r}\left( t^{2}-r^{2}\right) ^{s-1}t\,dt=\frac{1}{2}%
\int_{0}^{3r^{2}}\xi ^{s-1}d\xi =\frac{3^{s}}{2s}r^{2s},
\end{equation*}%
whence
\begin{equation*}
I_{2}\leq \phi (r)^{s}\int_{r}^{2r}\left( t^{2}-r^{2}\right) ^{s-1}t\,dt=%
\frac{3^{s}}{2s}r^{2s}\phi (r)^{s}.
\end{equation*}%
Combining the estimates of $I_{1}$ and $I_{2}$ we deduce (\ref{est-decr}).
\end{proof}

\begin{proof}[Proof of Corollary \protect\ref{cor-dx}]
We need to show that the condition
\begin{equation}
\int_{r_{0}}^{+\infty }\left[ \int_{r}^{+\infty }\frac{tdt}{\mu (B(o,t))}%
\right] ^{q-1}rdr<\infty ,  \label{cond-int1ab}
\end{equation}%
is equivalent to a simpler condition
\begin{equation}
\int_{r_{0}}^{+\infty }\frac{r^{2q-1}dr}{[\mu (B(o,r))]^{q-1}}<\infty .
\label{cond-int1bb}
\end{equation}%
Indeed, the implication (\ref{cond-int1ab})$\Rightarrow $(\ref{cond-int1bb})
follows trivially by reducing the domain $[r,\infty )$ of integration in (%
\ref{cond-int1ab}) to $\left[ r,2r\right] .$

The converse implication follows from the following inequality that holds
for any non-increasing function $\phi \geq 0$ and any $s>0$:
\begin{equation}
\int_{a}^{+\infty }\left( \int_{r}^{+\infty }\phi (t)\,t\,dt\right)
^{s}rdr\leq C(s)\,\int_{a}^{+\infty }\phi (t)^{s}t^{2s+1}dt.  \label{hardy}
\end{equation}%
Indeed, applying (\ref{hardy}) with $s=q-1$, $a=r_{0}$, and $\phi (t)=\frac{1%
}{\mu (B(o,t))}$, we see that (\ref{cond-int1bb}) yields (\ref{cond-int1ab}%
).

In the case $s\geq 1$, inequality (\ref{hardy}) holds for all
non-negative measurable functions $\phi $, and is known as Hardy's inequality
(see, for instance, \cite[Sec. 1.3.1]{M}). In the case $0<s<1$, (\ref{hardy}%
) for non-increasing functions $\phi $ follows from Lemma \ref{prop-s} by
integrating both sides of (\ref{est-decr}) with respect to $r dr$, which yields%
\begin{eqnarray*}
&&\int_{a}^{\infty }\left( \int_{r}^{\infty }\phi (t)\,t\,dt\right) ^{s}rdr
\notag \\
&\leq &C\int_{a}^{+\infty }\left( \int_{r}^{\infty }\phi
(t)^{s}\,t^{2s-1}\,dt\right) rdr+C\int_{a}^{+\infty }\,\phi (r)^{s}r^{2s+1}dr
\\
&=&C\int_{a}^{+\infty }\left( \int_{0}^{t}rdr\right) \phi
(t)^{s}\,t^{2s-1}\,dt+C\int_{a}^{+\infty }\phi (r)^{s}r^{2s+1}dr \\
&=&C\int_{a}^{+\infty }\phi (t)^{s}\,t^{2s+1}\,dt.
\end{eqnarray*}
\end{proof}

\end{document}